\documentclass[12pt,oneside]{amsart}

\usepackage[margin=1in]{geometry}

\usepackage{amsmath,amsthm,amscd,amssymb}

\usepackage{graphicx}
\usepackage{xcolor}
\usepackage{tikz}
\usepackage{pgfplots}
\usetikzlibrary{math,calc}
\pgfplotsset{compat=1.17}

\usepackage[
  colorlinks=true,
  linkcolor=red,
  citecolor=red,
  urlcolor=red
]{hyperref}

\numberwithin{equation}{section}

\theoremstyle{plain}
\newtheorem{theorem}{Theorem}[section]

\newtheorem{corollary}[theorem]{Corollary}
\newtheorem{proposition}[theorem]{Proposition}

\theoremstyle{definition}
\newtheorem{definition}[theorem]{Definition}

\theoremstyle{remark}
\newtheorem{remark}[theorem]{Remark}

\DeclareMathOperator{\supp}{supp}
\DeclareMathOperator{\FR}{FR}
\DeclareMathOperator{\spec}{spec}
\DeclareMathOperator*{\E}{{\mathbb E}}
\DeclareMathOperator{\tr}{tr}

\date{}

\author{A. Iosevich}
\address{Department of Mathematics, University of Rochester, Rochester, NY}
\email{iosevich@gmail.com}

\author{A. Mayeli}
\address{Department of PhD Program Mathematics, City University of New York, NY}
\email{amayeli@gc.cuny.edu}

\author{E. Wyman}
\address{Department of Mathematics, University of Binghamton, Binghamton, NY}
\email{emmett.Wyman@gmail.com}

\thanks{A.I. was supported in part by the National Science Foundation grant DMS-2506858. A.M. was supported in part by the National Science Foundation grant DMS-2453769, the AMS-Simons Research Enhancement Grant, and the PSC-CUNY research grants. E.W. was supported in part by the National Science Foundation under grant no. DMS-2422900.}

\begin{document}

\title[Spectral synthesis on Riemannian manifolds]{Spectral synthesis on Riemannian manifolds}

\begin{abstract}
We study spectral synthesis for measures supported on thin subsets of compact Riemannian manifolds. We prove that under natural non-concentration conditions, such measures admit quantitative spectral synthesis, with explicit stability bounds. 

We show that this phenomenon depends strongly on the underlying geometry. On the torus, synthesis holds under broad assumptions, while on the sphere we establish rigidity results demonstrating that synthesis can fail in a sharp sense. 

As consequences, we obtain quantitative approximation results and uncertainty principles for functions with thin spectral support. These results provide a unified framework connecting spectral synthesis, geometric structure, and stability on compact manifolds.
\end{abstract}

\maketitle

\section{Introduction}

Spectral synthesis is a central problem in harmonic analysis, concerned with the reconstruction of functions or measures from partial Fourier data. While the classical theory is well understood in Euclidean settings, much less is known on compact manifolds, where curvature and global geometry play a decisive role.

In this paper, we establish quantitative spectral synthesis results for measures supported on thin subsets of compact Riemannian manifolds. Our main theorem shows that under natural non-concentration assumptions, such measures admit stable reconstruction from their Fourier data. We further demonstrate that the validity of synthesis depends sharply on the underlying geometry: on the torus the phenomenon is robust, while on the sphere we prove rigidity results that exhibit a fundamental obstruction.

\vskip.125in

A celebrated result of Agranovsky and Narayanan \cite{AN04} asserts that if $f \in L^p(\mathbb R^d)$, $d \ge 2$, and $\widehat{f}$ is carried by a $C^1$ submanifold $S \subset \mathbb R^d$ of dimension $k$, $1 \le k < d$, then $f$ is the zero function provided that $1 \le p \le \frac{2d}{k}$. Analogous results were previously established by Agmon, H\"ormander, Thangavelu, and others; see, for example, \cite{AH76}, \cite{H83}, and \cite{T94}.

\vskip.125in

The purpose of this paper is to establish a compact-manifold analogue of the Agranovsky--Narayanan theorem and to highlight a range of phenomena that have no Euclidean counterpart. Let $(M,g)$ be a compact $d$-dimensional Riemannian manifold without boundary. Let $\Delta_g$ be the Laplace--Beltrami operator, and for each spectral parameter $\lambda$ let $E_\lambda$ denote the orthogonal projection onto the $-\lambda^2$ eigenspace of $\Delta_g$. When $u$ is a distribution on $M$ and
$\varphi\in C^\infty(M)$, we interpret $\int_M u\,\overline{\varphi}$ as the distributional
pairing $\langle u,\varphi\rangle$. When $u$ is a distribution on $M$, we define
\[
    E_\lambda u = \sum_{\lambda_j = \lambda} \left(\int_M u \overline{e_j}\right) e_j.
\]

Our first result is a rigidity statement for measures supported on thin sets.

\begin{theorem}\label{main}
Let $M$ be a compact Riemannian manifold without boundary, and let $u$ be a Radon measure carried by a $C^1$ submanifold of $M$ of dimension $k$, $1 \le k < d$. If
\[
    \sum_\lambda \|E_\lambda u\|_{L^2(M)}^p < \infty
\]
for some $p$ with $2 \le p \le \frac{2d}{k}$, then $u$ is identically $0$.
\end{theorem}
\noindent
For convenience, we write
\[
\|u\|_{\widehat\ell^p}:=\left(\sum_\lambda \|E_\lambda u\|_{L^2(M)}^p\right)^{1/p}.
\]

\begin{remark}\label{rem:Minkowski}
    The proof of Theorem~\ref{main} shows that the argument still goes through, up to the endpoint, if we assume that $f$ is a signed measure supported on a compact subset of $M$ of upper Minkowski dimension $k$. This is consistent with the Euclidean generalization of Agranovsky--Narayanan due to Senthil~Raani~\cite{SR14}. In particular, the endpoint is captured if the Riemannian volume of the $\delta$-neighborhood of the support satisfies $|E^{\delta}| \leq C \delta^{d-k}$ for all sufficiently small $\delta > 0$.

    \medskip

    \noindent
    Thus, Theorem~\ref{main} extends to sets of fractal nature under suitable volume growth conditions.
\end{remark}

Although we often formulate our results for functions $u\in L^2$, the arguments are not
intrinsically tied to this setting. In many instances, the same Fourier ratio bounds and
spectral estimates apply to finite measures, or more generally to distributions whose
spectral projections satisfy suitable growth conditions. We adopt the $L^2$ formulation
primarily for clarity and simplicity, but the framework allows for a broader class of
objects without essential changes to the analysis.

\subsection{Uniqueness and stability}
Theorem~\ref{main} is the uniqueness component of a more robust synthesis principle on compact manifolds. The next statement complements it with a quantitative stability bound, which may be interpreted as a ``signal recovery'' principle: a signal supported on a thin set cannot have small $\widehat\ell^p$ spectral content unless it is small in the sense of distributions.

\begin{theorem}[Quantitative stability for thinly supported signals]\label{thm:stability}
Let $u$ be a (complex, Radon) measure supported in a compact set $E\subset M$ with upper Minkowski dimension $k<d$. More precisely, assume that there exists a constant $C_E$ such that for all sufficiently small $\delta>0$,
\begin{equation}\label{eq:neigh-vol}
    |E^\delta|\le C_E\,\delta^{d-k},
\end{equation}
where $E^\delta$ denotes the $\delta$-neighborhood of $E$ in $M$ and $|\cdot|$ is Riemannian volume.

Fix $p$ with $2<p<\frac{2d}{k}$ and write
\[
\|u\|_{\widehat\ell^p}:=\left(\sum_\lambda \|E_\lambda u\|_{L^2(M)}^p\right)^{1/p}.
\]
Then there exists an even Schwartz function $\psi$ with $\widehat\psi$ supported in $[-1,1]$ and $\psi(0)=1$, and a constant $C=C(M,g,\psi,C_E,p)$, such that for every $R\ge 1$ and every test function $\chi\in C^\infty(M)$,
\begin{equation}\label{eq:stability-pairing}
\left|\left\langle P_Ru,\chi\right\rangle\right|
\le C\,\|u\|_{\widehat\ell^p}\,\|\chi\|_{L^\infty(M)}\,R^{\frac{k}{2}-\frac{d}{p}},
\end{equation}
where $P_R:=\psi(\sqrt{-\Delta_g}/R)$.
In particular, if $\|u\|_{\widehat\ell^p}$ is small, then all low-pass reconstructions $P_Ru$ are small uniformly on scales $R\gtrsim 1$, with an explicit rate in $R$.
\end{theorem}

\begin{remark}[Relation to signal recovery]
Although our focus here is on spectral synthesis and geometric rigidity on compact manifolds,
the quantitative stability estimate in Theorem~\ref{thm:stability} admits a natural interpretation as a
robust identifiability principle for inverse problems. Indeed, for a measure $u$ supported
on a thin set $E$ satisfying the neighborhood volume bound \eqref{eq:neigh-vol}, the theorem shows that the
low-frequency reconstructions $P_R u$ are uniformly small whenever the spectral $\ell^p$
mass of $u$ is small, with an explicit decay rate in $R$. Equivalently, within the class of
thinly supported signals, the map $u \mapsto P_R u$ is stable in the distributional sense
on all scales $R \gtrsim 1$.

This perspective suggests that the methods developed here should be well--suited to a range
of signal recovery problems on manifolds, for example, recovery from partial or coarse
spectral information under structural assumptions on the support. We do not pursue these
directions in the present paper.
\end{remark}

\subsection{Implications for compressed sensing on manifolds}
The quantitative stability estimate in Theorem~\ref{thm:stability} can be viewed as a deterministic precursor to compressed sensing guarantees on manifolds. In Euclidean compressed sensing, a signal that is sparse in space can be recovered from a few random Fourier measurements provided its Fourier transform satisfies suitable decay or incoherence conditions \cite{donoho2006compressed,candes2006near}.

In our setting, a measure $u$ supported on a thin set $E$ with $\|u\|_{\widehat\ell^p}$ small is, in essence, spatially sparse and spectrally compressible. Theorem~\ref{thm:stability} shows that its low-frequency projection~$P_R u$ is small, a property that suggests stability under incomplete spectral measurements.

More concretely, if one observes only a random subset of spectral coefficients $\{\langle u,e_j\rangle\}_{j\in J}$ where $J$ is a random set of eigenfunction indices, the estimate
\[
\bigl|\langle P_R u,\chi\rangle\bigr|
\lesssim
\|u\|_{\widehat\ell^p}\,R^{\frac{k}{2}-\frac{d}{p}}\,\|\chi\|_{L^\infty(M)}.
\]
This 
implies that the energy of $u$ cannot concentrate in high frequencies unless its support is thick, thereby limiting possible aliasing artifacts caused by undersampling.

The eigenfunction growth parameter~$A(\lambda)$ appearing in Theorem~\ref{theorem:localFRlowerbound} plays a role analogous to the coherence of the sensing dictionary in classical compressed sensing.
Manifolds with bounded $A(\lambda)$ (for example, flat tori) exhibit the most favorable incoherence, while spheres, where $A(\lambda)\sim \lambda^{(d-1)/2}$, present a more challenging high-coherence regime.
Thus, our results provide geometric criteria that inform the design of efficient sampling schemes on manifolds: thin spatial support together with moderate eigenfunction growth favors stable recovery from few spectral measurements.

A natural follow-up question is whether random undersampling of the spectral coefficients allows stable recovery via $\ell^1$--minimization, and how the required number of measurements depends on the dimension $k$ of the support and the growth $A(\lambda)$.
We plan to pursue this connection in future work.

\begin{remark}[Conceptual picture]\label{rem:conceptual-picture}
The results in this paper may be viewed as manifestations of a single synthesis principle: thin spatial support imposes strong restrictions on how spectral mass may be distributed. This appears as rigidity (Theorem~\ref{main}), as quantitative stability (Theorem~\ref{thm:stability}), and later as compressibility/approximation and uncertainty phenomena governed by the Fourier ratio. We do not attempt a general optimality theory beyond the sharp thresholds stated; rather, we aim to isolate robust mechanisms that persist across broad classes of compact manifolds.
\end{remark}

At a conceptual level, all results in this paper are governed by a single structural mechanism: quantitative control of how spectral mass may concentrate for distributions supported on geometrically thin sets. The central result is the spectral synthesis theorem for thinly supported measures (Theorem~\ref{main}), which identifies a sharp regime in which $\ell^p$ summability of spectral projections forces uniqueness. The stability estimate in Theorem~\ref{thm:stability} is a quantitative avatar of the same phenomenon, obtained by retaining explicit error terms in the synthesis argument. The approximation results by short spectral sums and the uncertainty principles involving the Fourier ratio reflect this same rigidity, reinterpreted in terms of compressibility rather than vanishing. The spherical examples isolate an extremal geometric setting in which this mechanism becomes maximally rigid, demonstrating that the critical behavior is genuinely geometric rather than a formal extension of Euclidean Fourier analysis.

\begin{figure}[ht]
\centering
\begin{tikzpicture}[scale=1.3, >=stealth]

  \begin{scope}[xshift=-3.2cm]
    \draw[thick, fill=blue!10] (0,0) ellipse (1.2 and 0.6);
    \draw[thick, fill=blue!10] (0,0.08) ellipse (1.2 and 0.6);
    \draw[thick] (1.2,0) arc (0:180:1.2 and 0.6);
    \draw[thick, densely dashed] (-1.2,0) arc (180:360:1.2 and 0.6);
    
    \draw[ultra thick, red!80!black] 
      plot[smooth, tension=0.7] coordinates {(-0.9,-0.2) (-0.4,0.3) (0.15,-0.08) (0.8,0.24)};
    \node[red!80!black, above, font=\small] at (0.15,-0.08) {$S$};
    
    \fill[red] (0.15,-0.08) circle (1.2pt);
    
    \draw[->, thick, green!50!black] (0.15,-0.08) -- ++(0.6,0.6) node[midway, above, sloped, font=\scriptsize] {$E_{\lambda}u$};
    \draw[->, thick, green!50!black] (-0.4,0.3) -- ++(-0.5,0.7) node[midway, above, sloped, font=\scriptsize] {$E_{\lambda}u$};
    
    \node[below, font=\small] at (0,-0.9) {Manifold $M$};
    \node[above, font=\small\bfseries] at (0,1.0) {Spatial support};
  \end{scope}

  \begin{scope}[xshift=0cm]
    \draw[->] (0,0) -- (2.5,0) node[right, font=\small] {$\lambda$};
    \draw[->] (0,0) -- (0,1.8) node[above, font=\small] {$\|E_{\lambda}u\|_{L^2}$};
    
    \foreach \x/\y in {0.2/1.6, 0.5/1.4, 0.8/1.2, 1.1/1.1, 1.4/1.0, 1.7/0.9, 2.0/0.85, 2.3/0.8}
      \draw[very thick, orange!80!red] (\x,0) -- (\x,\y);
    \node[orange!80!red, above, font=\scriptsize] at (1.4,1.0) {Case A};
    
    \foreach \x/\y in {0.2/0.8, 0.5/0.35, 0.8/0.12, 1.1/0.04, 1.4/0.015}
      \draw[very thick, blue!70!green] (\x,0) -- (\x,\y);
    \node[blue!70!green, above, font=\scriptsize] at (0.8,0.12) {Case B};
    
    \draw[thick, dashed] (0,0.7) -- (2.5,0.7) node[midway, above, font=\scriptsize] {$\ell^p$ threshold};
    
    \node[below, font=\small\bfseries] at (1.25,-0.3) {Spectral decay};
    \node[below, align=center, font=\scriptsize] at (1.25,-0.7) 
      {Fast $\ell^p$ decay\\ $\Rightarrow$ rigidity};
  \end{scope}

  \begin{scope}[xshift=3.2cm]
    \draw[->] (0,0) -- (2.5,0) node[right, font=\small] {$x$};
    \draw[->] (0,-1.0) -- (0,1.0);
    
    \draw[ultra thick, blue, smooth] 
      plot[domain=0:2.5, samples=80] (\x, {0.6*sin(deg(3*\x)) + 0.2*sin(deg(7*\x))});
    \node[blue, above, font=\small] at (1.0,0.8) {$f$};
    
    \draw[ultra thick, red!70!black, densely dashed, smooth] 
      plot[domain=0:2.5, samples=40] (\x, {0.55*sin(deg(3*\x))});
    \node[red!70!black, below, font=\scriptsize] at (1.8,-0.6) {$P$ (few modes)};
    
    \draw[->, thick, purple] (1.2,0.4) -- (1.2,-0.4);
    \node[purple, align=center, font=\scriptsize] at (1.2,0) 
      {Small $\mathrm{FR}(f)$\\ $\Rightarrow$\\ good approx.};
    
    \node[below, font=\small\bfseries] at (1.25,-1.2) {Fourier ratio};
  \end{scope}

  \begin{scope}[yshift=-2.2cm]
    \node[anchor=west, align=left, font=\small\bfseries] at (-4.5,0.2) {Key:};
    
    \draw[thick, red!80!black] (-4.2, 0) -- (-3.8, 0) node[right, black, font=\scriptsize] {Thin support $S$};
    \draw[thick, orange!80!red] (-4.2, -0.4) -- (-3.8, -0.4) node[right, black, font=\scriptsize] {Slow decay};
    \draw[thick, blue!70!green] (-4.2, -0.8) -- (-3.8, -0.8) node[right, black, font=\scriptsize] {Fast $\ell^p$ decay};
    \draw[thick, purple] (-4.2, -1.2) -- (-3.8, -1.2) node[right, black, font=\scriptsize] {$\mathrm{FR}(f)$};
    
    \node[align=center, font=\scriptsize] at (0, -0.2) 
      {Sphere:\\ maximal rigidity};
    \node[align=center, font=\scriptsize] at (0, -1.0) 
      {Torus:\\ sharp $p=\frac{2d}{k}$};
    
    \node[align=center, font=\scriptsize, text width=5cm] at (3.5, -0.6) 
      {Thin support + $\sum\|E_\lambda u\|^p<\infty$\\ $\Rightarrow u\equiv 0$ (Theorem~\ref{main})};
  \end{scope}

\end{tikzpicture}
\caption{Conceptual illustration of spectral synthesis on a compact Riemannian manifold $M$. 
\textbf{(Left)} A thinly supported measure $u$ on a submanifold $S \subset M$ and its spectral projections $E_\lambda u$. 
\textbf{(Middle)} Spectral mass $\|E_\lambda u\|_{L^2}$ vs. frequency $\lambda$: rapid $\ell^p$ decay (Case B) forces $u=0$ under $p \leq 2d/k$ (Theorem~\ref{main}). 
\textbf{(Right)} The Fourier ratio $\mathrm{FR}(f)$ controls approximation by short spectral sums (Theorem~\ref{theorem:manifoldfourierratio}). 
Bottom: key concepts and geometry-dependent sharpness.}
\label{fig:concept}
\end{figure}
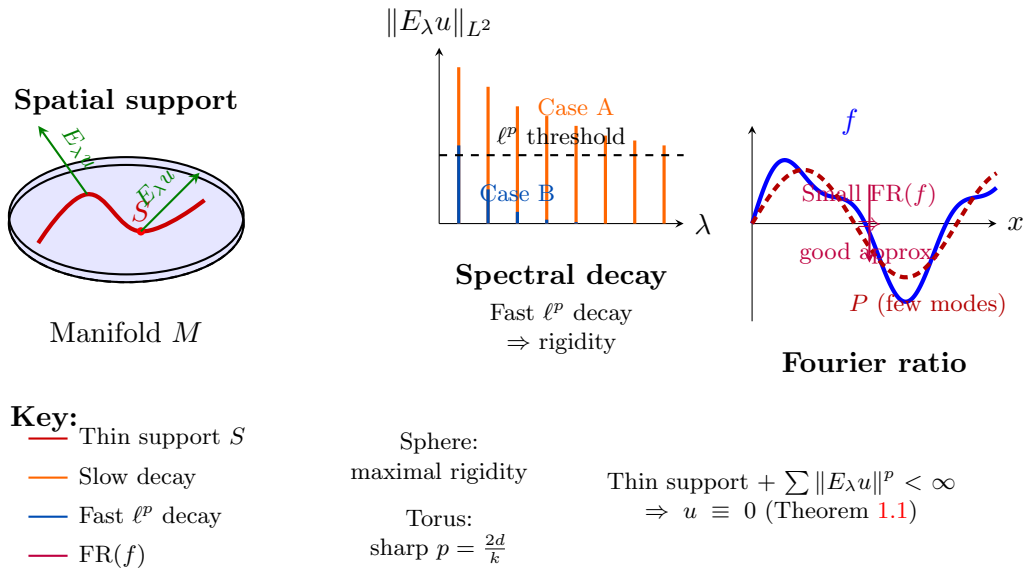

To illustrate the preceding discussion, consider the case $M=\mathbb T^d$, the $d$-dimensional torus. The eigenfunctions are the exponentials $e_j(x)=e^{2\pi i x\cdot j}$, so the indices $j$ contributing to the eigenvalue $-\lambda^2$ are
\[
\left\{j \in \mathbb Z^d:\ \sqrt{j_1^2+\cdots+j_d^2}=\lambda \right\}.
\]
Thus $E_\lambda$ projects a signal onto the frequencies lying on the sphere of radius $\lambda$, where $\lambda$ is the square root of a nonnegative integer. In this setting, the sharpness examples constructed by Agranovsky and Narayanan in the Euclidean case go through with little change; in sharp contrast, if eigenfunctions exhibit large $L^\infty$ growth (as they do on certain manifolds), the synthesis threshold may lose its Euclidean meaning.

\subsection{Sharpness and lack thereof}
Agranovsky and Narayanan \cite{AN04} proved that the exponent $\frac{2d}{k}$ is, in general, sharp if $S$ has dimension at least $\frac d2$. Similar examples can be constructed in the context of this paper if $\mathbb R^d$ is replaced by $\mathbb T^d$. When the dimension is smaller, this sharpness may fail. For example, if $S=\{(t,t^2,\dots,t^d): t\in[0,1]\}$ (or a small perturbation thereof), Guo, Iosevich, Zhang, and Zorin-Kranich \cite{GIZZ2024} showed that if $\widehat f$ is carried by a smooth measure on this moment curve and $f\in L^p(\mathbb R^d)$ for $p<\frac{d^2+d+2}{2}$, then $f$ is identically $0$. We are about to see that in the Riemannian setting, lack of sharpness in the classical estimate may result from the underlying geometry of the manifold.

\vskip.125in

In sharp contrast with the Euclidean and toral settings, the following result shows that on the sphere the synthesis phenomenon becomes maximally rigid: there is no finite sharp $L^p$ threshold.

\begin{proposition}[Lack of sharpness on the sphere]\label{prop:sphere-lack-sharpness}
Let $M = S^d$ be the $d$-dimensional sphere with the standard metric, and let $u$ be a $C^\infty$ surface-carried measure on a smooth embedded submanifold $S$. Then
\[
   \sum_\lambda \|E_\lambda u\|_{L^2(M)}^p < \infty \qquad \text{for some } p < \infty
\]
implies $u = 0$. Furthermore, if $u \neq 0$, then
\[
    \sup_\lambda \|E_\lambda u\|_{L^2(M)} < \infty
\]
if and only if $S$ is a hypersurface.
\end{proposition}

\subsection{Explicit example on $S^2$}

Consider $M=S^2$ with the standard metric, and let $S$ be the equator $\{ \theta = \pi/2 \}$.
Let $u = \delta_S$ be the normalized arclength measure on $S$. 
The spherical harmonics $Y_l^m$ ($l>0, |m|\leq l$) form an orthonormal basis of $L^2(S^2)$ (with respect to the standard surface measure) and are eigenfunctions of the Laplace-Beltrami operator:  $$\Delta_{S^2} Y_l^m = -l(l+1) Y_l^m.$$

A direct computation shows that only the $m=0$ harmonics contribute:
\[
\langle u, Y_l^m \rangle = 
\begin{cases}
2\pi \sqrt{\frac{2l+1}{4\pi}} P_l(0), & m=0, \\[4pt]
0, & m \ne 0,
\end{cases}
\]
where $P_l$ is the Legendre polynomial. Consequently,
\[
\|E_{\sqrt{l(l+1)}} u\|_{L^2}^2 = \pi(2l+1) |P_l(0)|^2.
\]
For even $l$, $|P_l(0)| \sim \sqrt{2/(\pi l)}$ as $l\to\infty$, hence
\[
\|E_{\sqrt{l(l+1)}} u\|_{L^2}  \to 2, \quad (\text{as} ~~ l\to \infty).
\]
Thus, $\|E_\lambda u\|_{L^2}$ is bounded below by a positive constant for infinitely many $\lambda$. 
Since the number of distinct eigenvalues $\lambda \le \Lambda$ grows linearly in $\Lambda$, the sum
\[
\sum_{\lambda} \|E_\lambda u\|_{L^2}^p
\]
diverges for every $p<\infty$, while $u\neq 0$. 
This illustrates the ``maximal rigidity'' stated in Proposition~\ref{prop:sphere-lack-sharpness}: on the sphere, no finite $p$ can replace the threshold $p\le 2d/k$ valid on the torus.

\vskip.125in 

The following is an immediate consequence of Proposition~\ref{prop:sphere-lack-sharpness}.

\begin{corollary}\label{cor:sphere}
Let $M = S^d$ be the $d$-dimensional sphere with the standard metric, and let $u$ be a $C^\infty$ surface-carried measure on a $C^\infty$ embedded $k$-dimensional submanifold $S$ with $1 \leq k < d$ and
\[
    \sum_\lambda \|E_\lambda u\|_{L^2(M)}^p < \infty
\]
for some $p<\infty$. Then, $u = 0$.
\end{corollary}

\begin{remark} \label{remark:regularity-sphere}
We do not know if it is reasonable to expect that Corollary~\ref{cor:sphere} can be extended to, say, $u \in L^1_{\mathrm{loc}}(S^d)$. The argument for the proposition uses Zelditch's Kuznecov asymptotic sum formula, which relies on the symbol calculus for the main term and on a stationary phase expansion to finite order. This results in some dependence on derivatives. We shall investigate this issue in detail in the sequel.
\end{remark}

\subsection{Spectral polynomial approximation}
Theorems~\ref{main} and \ref{thm:stability} show that small spectral integrability for a signal supported on a thin subset forces uniqueness and stability. Our next result shows that a small spectral ``Fourier ratio'' implies that the signal can be well-approximated in $L^2(M)$ by spectral polynomials of low complexity.

\vskip.125in

Let
\[
    \|f\|_{\widehat \ell^1} := \sum_\lambda \|E_\lambda f\|_{L^2(M)},
\qquad
    \FR(f) := \frac{\|f\|_{\widehat \ell^1}}{\|f\|_{L^2(M)}}.
\]

\begin{theorem}\label{theorem:manifoldfourierratio}
Let $f$ be a nontrivial $\widehat \ell^1$ function on $M$. Then for each $k$ satisfying
\[
k > \frac{\FR(f)^2 - 1}{\eta^2},
\]
there exists a linear combination $P$ of $k$ eigenfunctions (all from same eigenspace) on $M$ for which
\[
\|P - f\|_{L^2(M)} < \eta \|f\|_{L^2(M)}.
\]
\end{theorem}

The preceding theorem is one-directional: small $\FR(f)$ forces spectral approximability.
A partial converse can be stated cleanly under an a priori band-limitation hypothesis, which prevents the $\widehat\ell^1$ size of the error from being arbitrarily large.

\begin{definition}[Band-limited functions]\label{def:bandlimited}
Given $\Lambda>0$, we say $f\in L^2(M)$ is \emph{$\Lambda$-band-limited} if $E_\lambda f\equiv 0$ for all spectral parameters $\lambda>\Lambda$ (i.e., $f$ has no spectral components above $\Lambda$).
\end{definition}

\begin{theorem}[A band-limited converse bound for $\FR$]\label{thm:FR-converse-bandlimited}
Fix $\Lambda>0$ and suppose $f\in L^2(M)$ is $\Lambda$-band-limited and nontrivial. Assume there exists a linear combination $P$ of $k$ Laplace--Beltrami eigenfunctions such that
\begin{equation*}
\|f-P\|_{L^2(M)}\le \eta\,\|f\|_{L^2(M)}
\qquad\text{for some }0<\eta<1.
\end{equation*}
Let $N(\Lambda)=\#\{j\in\mathbb N:\lambda_j\le \Lambda\}$ be the Weyl counting function (with multiplicity).
Then
\begin{equation}\label{eq:FR-converse-bound}
\FR(f)
\le
(1+\eta)\sqrt{k}+\eta \sqrt{N(\Lambda)}.
\end{equation}
\end{theorem}

\begin{remark}
The inequality \eqref{eq:FR-converse-bound} is meaningful when $k\ll N(\Lambda)$ and $\eta$ is small: in that regime, a strong $k$-term approximation forces $\FR(f)$ to be $\lesssim \sqrt{k}$. In the absence of a band-limitation hypothesis, no bound of the form $\FR(f)\lesssim \sqrt{k}$ can hold for an arbitrary $L^2$ approximation error, since the $\widehat\ell^1$ size of the error may be arbitrarily large.
\end{remark}

\subsection{Properties of the Fourier ratio and the uncertainty principle}
We now prove lower and upper bounds on the Fourier ratio and show that it serves as a controlling parameter in a manifold variant of the classical Fourier uncertainty principle. We start with a lower bound.

\begin{theorem}[Scale-dependent uncertainty principle, band-limited version] \label{theorem:localFRlowerbound}
Suppose there exists a nondecreasing function $A:[0,\infty)\to(0,\infty)$ such that for every Laplace--Beltrami 
eigenfunction $\lambda$
\begin{equation}\label{eq:eig_growth_113}
\left(\sum_{\lambda_j=\lambda}|e_j(x)|^2\right)^{1/2}\le A(\lambda)
\end{equation}
for all $x\in M$ and all $\lambda\in\operatorname{spec}(\sqrt{-\Delta_g})$. Let $\psi\in\mathcal{S}(\mathbb{R})$ be even and assume $\widehat{\psi}$ has compact support. 
Let $C_\psi>0$ such that $\inf\{t: ~ |\psi(t)|>C_\psi\} >0$, 
and also fix a constant $c_\psi>0$ such that $|\psi(t)|\ge c_\psi$ for $|t|\le C_\psi$. For $R\ge 1$ define the spectral multiplier
\[
P_R=\psi\left(\frac{\sqrt{-\Delta_g}}{R}\right),
\qquad
P_R f=\sum_{\lambda}\psi(\lambda/R)E_\lambda f.
\]
There exist constants $c_0,C_0>0$ depending only on $(M,g)$ and $\psi$ such that the following holds. For every $f\in L^2(M)$ supported in a measurable set $E\subset M$ and every $R\ge 1$,
\begin{equation*}
\operatorname{supp}(P_R f)\subset E_{C_0/R}.
\end{equation*}
Assume in addition that $f$ is spectrally localized at scale $R$ in the sense that
\[
E_\lambda f=0 \qquad \hbox{whenever }|\lambda|>C_\psi R.
\]
Define the localized Fourier ratio at scale $R$ by
\begin{equation*}
\FR_R(f):=\frac{\sum_\lambda |\psi(\lambda/R)|\,\|E_\lambda f\|_{L^2(M)}}{\left(\sum_\lambda |\psi(\lambda/R)|^2\,\|E_\lambda f\|_{L^2(M)}^2\right)^{1/2}},
\end{equation*}
and define the eigenfunction growth parameter
\[
A_R:=\sup\{A(\lambda):\lambda\in\spec(\sqrt{-\Delta_g}),|\lambda|\le C_\psi R\}.
\]
Then $A_R<\infty$ and one has the lower bound
\begin{equation}\label{eq:lower_FRR_113}
\FR_R(f)\ge \frac{c_0}{A_R R^{d/2}\sqrt{|E_{C_0/R}|}}.
\end{equation}
\end{theorem}

\begin{remark}
In \eqref{eq:eig_growth_113} the quantity $\sum_{\lambda_j=\lambda}|e_j(x)|^2$ is
well defined: it is the same for every orthonormal basis $\{e_j\}$ of the
$-\lambda^2$ eigenspace. Indeed, if $E_\lambda$ is the orthogonal projection onto
this eigenspace and $\Pi_\lambda(x,y)$ is the Schwartz kernel of $E_\lambda$, then
\[
\Pi_\lambda(x,x)=\sum_{\lambda_j=\lambda}|e_j(x)|^2,
\]
so \eqref{eq:eig_growth_113} can be written in the basis-free form
\[
\Pi_\lambda(x,x)\le A(\lambda)^2 \qquad \text{for all } x\in M.
\]
Also, $\Pi_\lambda(x,x)$ can be viewed as a local density of states: integrating on
$M$ gives the multiplicity,
\[
\tr(E_\lambda)=\int_M \Pi_\lambda(x,x)\,dV_g(x)
=\dim \ker(\Delta_g+\lambda^2).
\]
\end{remark}

\begin{remark}
The quantitative bound in Theorem~\ref{theorem:localFRlowerbound} depends implicitly on the eigenfunction growth
assumption
\[
A(\lambda)\lesssim \lambda^{a}.
\]
In particular, the constants in the estimate deteriorate as the exponent $a$ increases.
Manifolds with improved eigenfunction bounds, therefore, yield correspondingly stronger
quantitative conclusions.
\end{remark}

\begin{remark}[Dependence on eigenfunction growth]
The factor $A_R R^{d/2}$ in the denominator of \eqref{eq:lower_FRR_113} reflects the combined effect of eigenfunction growth and spectral concentration. 
On manifolds where eigenfunctions are uniformly bounded in $L^\infty$ (e.g., flat tori), one has $A_R \lesssim 1$ and therefore $\FR_R(f) \gtrsim (R^{d/2}\sqrt{|E^{C_0/R}|})^{-1}$. 
On spheres, where the zonal harmonics saturate the H\"ormander bound $A(\lambda) \lesssim \lambda^{(d-1)/2}$, we have $A_R \lesssim R^{(d-1)/2}$, yielding 
$\FR_R(f) \gtrsim (R^{d-1/2}\sqrt{|E^{C_0/R}|})^{-1}$. 
The constant $c_0$ depends on $(M,g)$, $\psi$, and the implicit constant in the growth bound for $A(\lambda)$.
\end{remark}

\begin{remark}[Global vs.\ localized Fourier ratios]
The Fourier ratio $\FR(f)$ defined above and its localized counterpart $\FR_R(f)$ measure related but distinct notions of spectral concentration. The quantity $\FR(f)$ is global, involving the full spectrum of $\sqrt{-\Delta_g}$, and it controls approximation by short spectral sums as in Theorem~\ref{theorem:manifoldfourierratio}. In contrast, $\FR_R(f)$ captures spectral concentration at a fixed frequency scale $R$ and is adapted to scale-dependent uncertainty principles, as in Theorems~\ref{theorem:localFRlowerbound} and~\ref{theorem:manifold-FR-sandwich}. No limiting or monotonic relationship between $\FR(f)$ and $\FR_R(f)$ is assumed or required for the results of this paper.
\end{remark}

Our next result is an upper bound on the Fourier ratio. This bound is then combined with Theorem~\ref{theorem:localFRlowerbound} to produce a variant of the classical Fourier uncertainty principle.

\begin{theorem}\label{theorem:manifold-FR-sandwich}
With the notation of Theorem \ref{theorem:localFRlowerbound}, there exist constants $c_0, C_0 > 0$ (depending only on $(M,g)$ and $\psi$) such that the following hold.
\begin{enumerate}
\item[\textup{(i)}]
Let $0 < \eta < 1$ and let $\Sigma_R \subset \spec(\sqrt{-\Delta_g})$ be a subset of spectral parameters. Assume the $L^1$ spectral concentration condition
\[
\sum_{\substack{\lambda:\,|\psi(\lambda/R)| \ne 0 \\ \lambda \notin \Sigma_R}}
    |\psi(\lambda/R)| \|E_\lambda f\|_{L^2(M)}
\le
\eta \sum_{\lambda} |\psi(\lambda/R)| \|E_\lambda f\|_{L^2(M)}.
\]
Let
\[
M_R(\Sigma_R) = \#\left\{ \lambda \in \Sigma_R : |\psi(\lambda/R)| \ne 0 \right\}
\]
denote the number of localized spectral parameters in $\Sigma_R$. Then
\[
\FR_R(f) \le \frac{\sqrt{M_R(\Sigma_R)}}{1-\eta}.
\]

\item[\textup{(ii)}]
Combining Theorem \ref{theorem:localFRlowerbound} and \textup{(i)} yields
\[
M_R(\Sigma_R)\,|E^{C_0/R}| \ge \frac{c_0^2}{A_R^2 R^d} (1-\eta)^2.
\]
\end{enumerate}

In particular, for fixed $\eta \in (0,1)$ and bounded eigenfunction growth on the window (that is, $A_R \lesssim 1$), the product of the spatial volume of the $C_0/R$-thickening of $\supp f$ and the number of active spectral parameters in $\Sigma_R$ cannot be arbitrarily small.
\end{theorem}

We collect the geometric preliminaries and standard spectral facts used throughout the paper in Appendix~\ref{app:geometry}. We now proceed to the proofs of the main results.

\section{Proofs of the main results}

We prove Theorem~\ref{thm:stability} first (stability), then deduce  Theorem~\ref{main}  for $2<p<2d/k$, and finally handle the endpoint $p=2d/k$ via a dyadic decomposition.

\subsection{Proof of Theorem \ref{thm:stability}}

\begin{proof}
Let $\psi$ be an even Schwartz function with $\widehat\psi$ supported in $[-1,1]$ and $\psi(0)=1$, and let
\[
P_R\mu=\sum_\lambda \psi(\lambda/R)\,E_\lambda \mu.
\]
Since $\psi$ is even and $\widehat\psi$ is compactly supported, we may write
\[
\psi(\lambda)=\frac{1}{2\pi}\int_{\mathbb R}\cos(t\lambda)\,\widehat\psi(t)\,dt,
\]
and hence, in the sense of functional calculus,
\[
P_R=\frac{R}{2\pi}\int_{\mathbb R}\widehat\psi(Rt)\,\cos\bigl(t\sqrt{-\Delta_g}\bigr)\,dt.
\]
By finite propagation speed for the wave equation on $(M,g)$, there exists a constant $C=C(M,g,\psi)>0$ such that if $\supp \mu \subset E$, then
\[
\supp(P_R\mu)\subset E^{C/R}.
\]

By orthogonality and H\"older with exponents  $q=\frac{p}{p-2}$ and $q'=\frac{p}{2}$ for $p>2$ (so $1-2/p>0$) we have 

\begin{align*}
\|P_R\mu\|_{L^2(M)}^2
&=
\sum_\lambda |\psi(\lambda/R)|^2\,\|E_\lambda\mu\|_{L^2(M)}^2 \\
&\le
\left(\sum_\lambda |\psi(\lambda/R)|^{\frac{2}{1-\frac{2}{p}}}\right)^{1-\frac{2}{p}}
\left(\sum_\lambda \|E_\lambda\mu\|_{L^2(M)}^p\right)^{2/p}.
\end{align*}
By the Weyl law \eqref{eq:Weyl-law} and the rapid decay of $\psi$, the first factor is $\lesssim R^{d(1-\frac{2}{p})}$ (using the counting function $N(\Lambda)\lesssim \Lambda^d$ and comparing  the sum to $\int \psi(t)^q dN(Rt)$), hence
\begin{equation}\label{eq:PR-L2-stability}
\|P_R\mu\|_{L^2(M)}\lesssim R^{\frac{d}{2}-\frac{d}{p}}\|\mu\|_{\widehat\ell^p}.
\end{equation}

Now fix $\chi\in C^\infty(M)$. Using $\|\chi\|_{L^2(\supp P_R\mu)}\le \|\chi\|_\infty\,|\supp(P_R\mu)|^{1/2}$, we obtain
\[
|\langle P_R\mu,\chi\rangle|
\le
\|P_R\mu\|_{L^2(M)}\,\|\chi\|_{L^2(\supp P_R\mu)}
\le
\|P_R\mu\|_2\,\|\chi\|_\infty\,|\supp(P_R\mu)|^{1/2}.
\]
By \eqref{eq:neigh-vol} and $\supp(P_R\mu)\subset E^{C/R}$, we have
\[
|\supp(P_R\mu)|\le |E^{C/R}|\le C_E\,(C/R)^{d-k}\lesssim R^{k-d}.
\]
Combining this with \eqref{eq:PR-L2-stability} yields
\[
|\langle P_R\mu,\chi\rangle|
\lesssim
\|\mu\|_{\widehat\ell^p}\,\|\chi\|_\infty\,R^{\frac{d}{2}-\frac{d}{p}}\cdot R^{\frac{k-d}{2}}
=
\|\mu\|_{\widehat\ell^p}\,\|\chi\|_\infty\,R^{\frac{k}{2}-\frac{d}{p}}.
\]
This is exactly \eqref{eq:stability-pairing}.
\end{proof}

\subsection{Proof of Theorem \ref{main}}

We first prove Theorem \ref{main} in the range $2<p<\frac{2d}{k}$ under the neighborhood volume hypothesis \eqref{eq:neigh-vol}. The endpoint argument is handled separately afterward.

\begin{proposition}\label{prop:main-nonendpoint}
Let $u$ be a (complex, Radon) measure on $M$ supported in a compact set $E\subset M$ satisfying \eqref{eq:neigh-vol} for some $k<d$, and suppose
\[
\sum_\lambda \|E_\lambda u\|_{L^2(M)}^p<\infty
\]
for some $p$ with $2<p<\frac{2d}{k}$. Then $u=0$.
\end{proposition}

\begin{proof}
Fix $\psi$ as in Theorem \ref{thm:stability}, and define $P_R=\psi(\sqrt{-\Delta_g}/R)$. For any test function $\chi\in C^\infty(M)$, Theorem \ref{thm:stability} gives
\[
|\langle P_Ru,\chi\rangle|
\le C\,\|u\|_{\widehat\ell^p}\,\|\chi\|_{L^\infty(M)}\,R^{\frac{k}{2}-\frac{d}{p}}.
\]
Since $p<\frac{2d}{k}$, we have $\frac{k}{2}-\frac{d}{p}<0$, so the right-hand side tends to $0$ as $R\to\infty$.

On the other hand, since $\psi(0)=1$ and $\psi$ is Schwartz, $P_Ru\to u$ weakly as $R\to\infty$. Therefore $\langle u,\chi\rangle=\lim_{R\to\infty}\langle P_Ru,\chi\rangle=0$ for every smooth $\chi$, hence $u=0$.
\end{proof}

\begin{proposition}\label{prop:main-endpoint}
Let $u$ be a (complex, Radon) measure on $M$ supported in a set satisfying \eqref{eq:neigh-vol} for some $k < d$, and suppose
\[
    \sum_\lambda \|E_\lambda u\|_{L^2(M)}^p < \infty
\]
for some $p$ with $2 < p \leq \frac{2d}{k}$. Then, $u = 0$.
\end{proposition}

This proposition, together with Proposition~\ref{prop:main-nonendpoint}, implies Theorem~\ref{main} as stated for the range $2\le p\le \frac{2d}{k}$.

\begin{proof}
Let $p_0 = \frac{2d}{k}$ be the endpoint. Let $\psi$ be as in Theorem \ref{thm:stability}, and define
\[
P_R u=\sum_\lambda \psi(\lambda/R)\,E_\lambda u.
\]
As in the proof of Theorem \ref{thm:stability}, $\supp(P_Ru)\subset E^{C/R}$ for some $C=C(M,g,\psi)>0$, and $P_Ru\to u$ weakly as $R\to\infty$.

We perform a dyadic decomposition:
\begin{align*}
\|P_R u\|_{L^2(M)}^{2}
&= \sum_\lambda |\psi(\lambda/R)|^2 \|E_\lambda u\|_{L^2(M)}^2 \\
&= |\psi(0)|^2 \left| \int_M u \right|^2
+ \sum_{j \in \mathbb Z}\,  \sum_{2^j < \lambda/R \leq 2^{j+1}} |\psi(\lambda/R)|^2 \|E_\lambda u\|_{L^2(M)}^2,
\end{align*}
where $\lambda$ is summed over the spectrum of $\sqrt{-\Delta_g}$, and where we have taken the $\lambda = 0$ terms out explicitly. Since the spectrum of $\sqrt{-\Delta_g}$ is discrete, the sum over $j$ is supported on a set of integers bounded below. For the dyadic part, write
\begin{align*}
\sum_{2^j < \lambda/R \leq 2^{j+1}}&(2^{j}R)^{d-k} |\psi(\lambda/R)|^2 (2^jR)^{k-d} \|E_\lambda u\|_{L^2(M)}^2 \\
&\leq \left( \sup_{\tau \in (2^j, 2^{j+1}]} (2^{j}R)^{d-k} |\psi(\tau)|^2 \right) \sum_{2^j < \lambda/R \leq 2^{j+1}} (2^jR)^{k-d} \|E_\lambda u\|_{L^2(M)}^2 \\
&= R^k a_j b_j(R),
\end{align*}
where
\[
a_j = \sup_{\tau \in (2^j, 2^{j+1}]} 2^{j(d-k)} |\psi(\tau)|^2
\qquad \text{and} \qquad
b_j(R) = \sum_{2^j < \lambda/R \leq 2^{j+1}} (2^jR)^{k-d} \|E_\lambda u\|_{L^2(M)}^2.
\]
By H\"older,
\begin{align*}
b_j(R)
&\leq (2^j R)^{k-d} \left( \sum_{R2^j < \lambda \leq R2^{j+1}} 1 \right)^{1 - \frac 2 {p_0}}
\left( \sum_{2^j <\lambda/R \leq 2^{j+1}} \|E_\lambda u\|_{L^2(M)}^{p_0} \right)^\frac 2 {p_0} \\
&= (2^j R)^{k-d} (N(2^{j+1}R) - N(2^j R))^{1 - \frac 2 {p_0}}
\left( \sum_{2^j < \lambda/R \leq 2^{j+1}} \|E_\lambda u\|_{L^2(M)}^{p_0} \right)^\frac 2 {p_0},
\end{align*}
where $N(\lambda)$ is the Weyl counting function. By the Weyl law \eqref{eq:Weyl-law}, $N(2^{j+1} R) \lesssim (2^j R)^d$, hence
\[
b_j(R) \lesssim (2^j R)^{d \left( 1 - \frac 2 {p_0} \right) - (d - k)}
\left( \sum_{2^j < \lambda/R \leq 2^{j+1}} \|E_\lambda u\|_{L^2(M)}^{p_0} \right)^\frac 2 {p_0}.
\]
Since $p_0 = \frac{2d}{k}$, we have $d(1 - 2/p_0) - (d - k) = 0$, so
\[
b_j(R) \lesssim
\left( \sum_{2^j < \lambda/R \leq 2^{j+1}} \|E_\lambda u\|_{L^2(M)}^{p_0} \right)^\frac 2 {p_0}.
\]
Since $\sum_\lambda \|E_\lambda u\|_{L^2(M)}^{p_0}<\infty$, it follows that $b_j(R)\to 0$ as $R\to\infty$ for each fixed $j$, and $|b_j(R)|$ are uniformly bounded in both $R$ and $j$. Since $\psi$ is Schwartz, $a_j$ decays rapidly in $j$ and $\sum_{j\in\mathbb Z} a_j<\infty$. By dominated convergence, and since $d-k>0$,
\[
R^{k-d}\|P_Ru\|_{L^2(M)}^2
=
R^{k-d}|\psi(0)|^2\left|\int_M u\right|^2
+\sum_j a_j b_j(R) \to 0.
\]
If $\chi$ is any smooth test function on $M$,
\[
|\langle P_Ru,\chi\rangle|^2
\le \|P_Ru\|_{L^2(M)}^2 \, |\supp(P_Ru)|
\lesssim \|P_Ru\|_{L^2(M)}^2 \, R^{k-d}\to 0,
\]
where we used $|\supp(P_Ru)|\le |E^{C/R}|\lesssim R^{k-d}$ from \eqref{eq:neigh-vol}. Finally, since $P_Ru\to u$ weakly as $R\to\infty$, we obtain $\langle u,\chi\rangle=0$ for all smooth $\chi$, hence $u=0$.
\end{proof}

Here we handle the case p=2: If $\sum_\lambda \|E_\lambda u\|_2^2<\infty$, then the spectral coefficients are in $\ell^2$. Therefore, $u$ defines an $L^2$ function through an eigenbasis. Since any $L^2$ function supported in a submanifold of dimension $k<d$ must be zero, we get $u=0$,  since the sub-manifold has $d$-volume zero. This completes the proof of the case $p=2$.

\subsection{Proof of Proposition \ref{prop:sphere-lack-sharpness}}
\begin{proof}
Let $u = f \delta_S$, where $f$ is a smooth function and $\delta_S$ is the induced volume measure on $S$. Then,
\[
\|E_\lambda u\|_{L^2}^2 = \sum_{\lambda_j = \lambda} \left| \int_S f \overline{e_j} \right|^2,
\]
where $\{ e_j : j \in \mathbb N \}$ is an orthonormal basis of Laplace--Beltrami eigenfunctions with
\[
\Delta_g e_j = -\lambda_j^2 e_j, \qquad \lambda_j \geq 0.
\]
A very general result by Zelditch \cite{steve}, yields asymptotics
\begin{equation}\label{eq:steve}
\sum_{\lambda_j \leq \lambda} \left| \int_S f \overline{e_j} \right|^2 = C \lambda^{d - k} + O(\lambda^{d - k - 1}),
\end{equation}
where the leading constant $C > 0$ is determined by $S$, $f$, and $M$.

The preceding setup applies to general manifolds; now we specialize to the sphere $S^d$. The spectrum of the Laplace--Beltrami operator, or more precisely $\sqrt{-\Delta_g}$, on the sphere is
\[
\left\{\sqrt{n(n+d-1)} : n \in \mathbb Z, \ n \geq 0 \right\}.
\]
Fix $T$ sufficiently large, and note that, for sufficiently large $\lambda$,
\begin{align*}
\sum_{\lambda_j \in (\lambda, \lambda + T]} \left| \int_S f \overline e_j \right|^2
&= C((\lambda + T)^{d-k} - \lambda^{d-k}) + O(\lambda^{d-k-1}) \\
&= C T \lambda^{d-k-1} + O(\lambda^{d-k-1}) \\
&\gtrsim \lambda^{d-k-1}.
\end{align*}
There are at most $T$ points in the spectrum of $\sqrt{-\Delta_g}$ lying in the interval $(nT, (n+1)T]$, and hence by pigeonholing, there exists $\lambda' \in (nT, (n+1)T]$ for which
\[
\|E_{\lambda'} u\|_{L^2}^2 = \sum_{\lambda_j = \lambda'} \left| \int_S f \overline e_j \right|^2 \gtrsim \frac 1 T \lambda^{d-k-1}.
\]
This readily yields the first claim of the proposition. The second claim follows after noting
\[
\sum_{\lambda_j \leq \lambda} \left| \int_S f \overline e_j \right|^2
= \sum_{\lambda' \le \lambda}\|E_{\lambda'}u\|_{L^2}^2
\le \left(\sup_{\lambda' \leq \lambda} \|E_{\lambda'} u\|_{L^2}^2\right)\#\{\lambda' \in \spec(\sqrt{-\Delta_g}) : \lambda' \le \lambda\},
\]
and on the sphere $\#\{\lambda' \in \spec(\sqrt{-\Delta_g}) : \lambda' \le \lambda\}\simeq \lambda$ (distinct spectral parameters grow linearly), giving
\[
\sum_{\lambda_j \leq \lambda} \left| \int_S f \overline e_j \right|^2
\lesssim \left( \sup_{\lambda'} \|E_{\lambda'} u\|_{L^2}^2 \right) \lambda,
\qquad \lambda \gg 1.
\]
Comparing with \eqref{eq:steve} forces $d-k\le 1$, i.e.\ $k\ge d-1$, hence $S$ is a hypersurface. This proves necessity; sufficiency follows by noting, as a consequence of \eqref{eq:steve} in the case $d - k = 1$,
\[
    \|E_\lambda u\|_{L^2}^2 \leq \sum_{\lambda_j \in (\lambda - 1, \lambda]} \left| \int_S f \overline e_j \right|^2 = O(1).
\]
\end{proof}

\subsection{Proof of Theorem \ref{theorem:manifoldfourierratio}}
\begin{proof}
Consider a random eigenfunction
\[
Z(x) = \frac{\|f\|_{\widehat \ell^1}}{\|E_\lambda f\|_{L^2}} E_\lambda f(x)
\]
where $\lambda$ is drawn from the spectrum of $\sqrt{-\Delta_g}$ randomly with probability $\|E_\lambda f\|_{L^2}/\|f\|_{\widehat \ell^1}$. Note
\[
\E Z(x) = f(x)
\]
and
\begin{align*}
\E |Z(x)|^2
&= \sum_{\lambda} \frac{\|E_\lambda f\|_2}{\|f\|_{\widehat \ell^1}} \left| \frac{\|f\|_{\widehat \ell^1}}{\|E_\lambda f\|_{L^2}} E_\lambda f(x) \right|^2 \\
&= \sum_\lambda \frac{\|f\|_{\widehat \ell^1}}{\|E_\lambda f\|_2} |E_\lambda f(x)|^2.
\end{align*}
Hence,
\[
\operatorname{Var}(Z(x)) = \E|Z(x)|^2 - |\E Z(x)|^2 = \sum_\lambda \frac{\|f\|_{\widehat \ell^1}}{\|E_\lambda f\|_2} |E_\lambda f(x)|^2 - |f(x)|^2.
\]

Consider $Z_1, \ldots, Z_k$ i.i.d.\ random variables with the same distribution as $Z$, and set
\[
P = \frac 1 k \sum_{i = 1}^k Z_i.
\]
Note $\E P(x) = f(x)$ for each $x \in M$, and furthermore,
\begin{align*}
\E \|P - f\|_2^2
&= \int_M \E |P(x) - f(x)|^2 \, dV_g(x) \\
&= \int_M \operatorname{Var} (P(x)) \, dV_g(x) \\
&= \frac 1 k \int_M \operatorname{Var}(Z(x)) \, dV_g(x) \\
&= \frac 1 k \int_M \left( \sum_\lambda \frac{\|f\|_{\widehat \ell^1}}{\|E_\lambda f\|_2} |E_\lambda f(x)|^2 - |f(x)|^2 \right) \, dV_g(x) \\
&= \frac 1 k \left( \|f\|_{\widehat \ell^1} \sum_\lambda \|E_\lambda f\|_2 - \|f\|_2^2 \right) \\
&= \frac 1 k ( \|f\|_{\widehat \ell^1}^2 - \|f\|_2^2) \\
&= \frac 1 k \|f\|_2^2 (\FR(f)^2 - 1).
\end{align*}
Hence, there exists a deterministic linear combination $P$ for which
\[
\|P - f\|_2^2 \leq \frac 1 k \|f\|_2^2 \cdot (\FR(f)^2 - 1).
\]
With $k$ as in the statement of the theorem, we have
\[
\|P - f\|_2 < \eta \|f\|_2
\]
as needed.
\end{proof}

\subsection{Proof of Theorem \ref{thm:FR-converse-bandlimited}}

\begin{proof}
Write $r=f-P$. Since $f$ is $\Lambda$-band-limited and $P$ is a finite linear combination of eigenfunctions, both $P$ and $r$ are also $\Lambda$-band-limited (after replacing $P$ by its projection onto the span of $\{\lambda_j\le \Lambda\}$, which does not increase the error).

First, by Cauchy--Schwarz on the $k$ orthonormal eigenfunctions entering $P$, we have
\begin{equation*}
\|P\|_{\widehat\ell^1}
=
\sum_\lambda \|E_\lambda P\|_{2}
\le
\sqrt{k}\left(\sum_\lambda \|E_\lambda P\|_2^2\right)^{1/2}
=
\sqrt{k}\,\|P\|_2.
\end{equation*}
Next, since $r$ is $\Lambda$-band-limited, the set of indices $\{j:\lambda_j\le \Lambda\}$ has size $N(\Lambda)$ (counted with multiplicity), and we obtain
\begin{equation*}
\|r\|_{\widehat\ell^1}
=
\sum_{\lambda\le \Lambda} \|E_\lambda r\|_2
\le
\sqrt{N(\Lambda)}\left(\sum_{\lambda\le \Lambda} \|E_\lambda r\|_2^2\right)^{1/2}
=
\sqrt{N(\Lambda)}\,\|r\|_2.
\end{equation*}
Therefore,
\[
\|f\|_{\widehat\ell^1}
\le
\|P\|_{\widehat\ell^1}+\|r\|_{\widehat\ell^1}
\le
\sqrt{k}\,\|P\|_2+\sqrt{N(\Lambda)}\,\|r\|_2.
\]
Using $\|r\|_2\le \eta \|f\|_2$ and $\|P\|_2\le \|f\|_2+\|r\|_2\le (1+\eta)\|f\|_2$, we obtain
\[
\|f\|_{\widehat\ell^1}
\le
\left((1+\eta)\sqrt{k}+\eta\sqrt{N(\Lambda)}\right)\|f\|_2,
\]
which is exactly \eqref{eq:FR-converse-bound}.
\end{proof}

\vskip.25in

\subsection{Proof of Theorem \ref{theorem:localFRlowerbound}}

\begin{proof}
Since $\psi$ is even and $\widehat{\psi}$ is compactly supported, the functional calculus gives the representation
\[
P_R=\psi\left(\frac{\sqrt{-\Delta_g}}{R}\right)=\frac{R}{2\pi}\int_{\mathbb{R}}\widehat{\psi}(Rt)\cos\left(t\sqrt{-\Delta_g}\right)\,dt.
\]
Finite propagation speed for the wave equation on $(M,g)$ implies that there exists a constant $C_0>0$, depending only on $(M,g)$ and $\operatorname{supp}\widehat{\psi}$, such that if $\operatorname{supp}f\subset E$, then $\operatorname{supp}(\cos(t\sqrt{-\Delta_g})f)\subset E^{C_0|t|}$ for every $t$. Since $\widehat{\psi}(Rt)$ vanishes unless $|t|\le T_0/R$ for a constant $T_0$ depending only on $\psi$, the integral representation yields
\[
\operatorname{supp}(P_R f)\subset E^{C_0/R}.
\]

We now estimate $\|P_R f\|_{L^\infty(M)}$ in terms of the spectral pieces of $f$. Writing the spectral decomposition,
\[
P_R f(x)=\sum_{\lambda}\psi(\lambda/R)E_\lambda f(x),
\]
and expanding $E_\lambda f$ in an orthonormal eigenbasis $\{e_j\}$ of the $\lambda$--eigenspace, we have
\[
E_\lambda f(x)=\sum_{\lambda_j=\lambda}\langle f,e_j\rangle_{L^2(M)}e_j(x).
\]
Cauchy--Schwarz gives
\[
|E_\lambda f(x)|
\le
\left(\sum_{\lambda_j=\lambda}|\langle f,e_j\rangle_{L^2(M)}|^2\right)^{1/2}
\left(\sum_{\lambda_j=\lambda}|e_j(x)|^2\right)^{1/2}.
\]
By orthogonality of the eigenspaces, the first factor equals $\|E_\lambda f\|_{L^2(M)}$, and by hypothesis \eqref{eq:eig_growth_113} the second factor is bounded by $A(\lambda)$. Therefore
\[
|E_\lambda f(x)|\le A(\lambda)\|E_\lambda f\|_{L^2(M)}.
\]
Substituting into the spectral expansion for $P_R f$ yields
\[
|P_R f(x)|\le \sum_{\lambda}|\psi(\lambda/R)|\,A(\lambda)\,\|E_\lambda f\|_{L^2(M)}.
\]
By the band-limitation hypothesis $E_\lambda f=0$ when $|\lambda|>C_\psi R$, and on the remaining window $A(\lambda)\le A_R$ by definition of $A_R$. Hence
\[
|P_R f(x)|\le A_R\sum_{\lambda}|\psi(\lambda/R)|\,\|E_\lambda f\|_{L^2(M)}.
\]
Set
\[
\operatorname{Num}_R(f):=\sum_{\lambda}|\psi(\lambda/R)|\,\|E_\lambda f\|_{L^2(M)}.
\]
Then
\[
\|P_R f\|_{L^\infty(M)}\le A_R \cdot \operatorname{Num}_R .
\]

We next record the scale-dependent estimate for $\operatorname{Num}_R(f)$ coming from Weyl's law. By Cauchy--Schwarz,
\[
\operatorname{Num}_R(f) \le \left(\sum_{\lambda}|\psi(\lambda/R)|^2\right)^{1/2}\left(\sum_{\lambda}\|E_\lambda f\|_{L^2(M)}^2\right)^{1/2}.
\]
The second factor equals $\|f\|_{L^2(M)}$ by orthogonality of the spectral projections. For the first factor, since $\psi$ is Schwartz there exists a constant $C_1>0$, depending only on $\psi$, such that $|\psi(t)|\le 1$ for all $t$ and this implies
\[
\sum_{\lambda}|\psi(\lambda/R)|^2
\le
\#\{\lambda\in\spec(\sqrt{-\Delta_g}):|\lambda|\le C_1 R\}.
\]
By the Weyl bound, there exists a constant $C_2>0$, depending only on $(M,g)$, such that
\[
\#\{\lambda\in\spec(\sqrt{-\Delta_g}):|\lambda|\le T\}\le C_2 T^d
\]
for all $T\ge 1$. Applying this with $T=C_1 R$ gives
\[
\sum_{\lambda}|\psi(\lambda/R)|^2\le C_2(C_1 R)^d.
\]
Consequently there exists a constant $C_3>0$, depending only on $(M,g)$ and $\psi$, such that
\[
\left(\sum_{\lambda}|\psi(\lambda/R)|^2\right)^{1/2}\le C_3 R^{d/2},
\]
and therefore
\[
\operatorname{Num}_R(f) \le C_3 R^{d/2}\|f\|_{L^2(M)}.
\]

We now use the band-limitation hypothesis together with the lower bound on $|\psi|$ on the active window to compare $\|f\|_{L^2(M)}$ and $\operatorname{Num}_R$ from below. Since $E_\lambda f=0$ for $|\lambda|>C_\psi R$ and $|\psi(t)|\ge c_\psi$ for $|t|\le C_\psi$, we have
\[
\operatorname{Num}_R(f)=\sum_{|\lambda|\le C_\psi R}\left|\psi\left(\frac{\lambda}{R}\right)\right|\,\|E_\lambda f\|_{L^2(M)}
\ge
c_\psi\sum_{|\lambda|\le C_\psi R}\|E_\lambda f\|_{L^2(M)}.
\]
Since the $\ell^1$ norm dominates the $\ell^2$ norm,
\[
\sum_{|\lambda|\le C_\psi R}\|E_\lambda f\|_{L^2(M)}
\ge
\left(\sum_{|\lambda|\le C_\psi R}\|E_\lambda f\|_{L^2(M)}^2\right)^{1/2}.
\]
By orthogonality of the projections and band-limitation, the right-hand side equals $\|f\|_{L^2(M)}$, hence
\[
\operatorname{Num}_R(f) \ge c_\psi\|f\|_{L^2(M)},
\qquad
\|f\|_{L^2(M)}\le  c_\psi^{-1}\operatorname{Num}_R.
(f) \]
Combining this with $\operatorname{Num}_R\le C_3 R^{d/2}\|f\|_{L^2(M)}$ and $\|P_R f\|_{L^\infty(M)}\le A_R \operatorname{Num}_R(f)$ gives
\[
\|P_R f\|_{L^\infty(M)}\le C_3 A_R R^{d/2}\operatorname{Num}_R(f).
\]

Finally, we use the support property of $P_R f$ to relate $\|P_R f\|_{L^2(M)}$ and $\|P_R f\|_{L^\infty(M)}$. Since $\operatorname{supp}(P_R f)\subset E^{C_0/R}$, we have
\[
\|P_R f\|_{L^2(M)}^2
=
\int_{E^{C_0/R}}|P_R f(x)|^2\,dV_g(x)
\le
|E^{C_0/R}|\|P_R f\|_{L^\infty(M)}^2,
\]
and therefore
\[
\|P_R f\|_{L^2(M)}
\le
\sqrt{|E^{C_0/R}|}\,\|P_R f\|_{L^\infty(M)}
\le
C_3 A_R R^{d/2}\sqrt{|E^{C_0/R}|}\,\operatorname{Num}_R
(f)\]
On the other hand, the orthogonality of the spectral projections gives
\[
\|P_R f\|_{L^2(M)}^2=\sum_{\lambda}|\psi(\lambda/R)|^2\|E_\lambda f\|_{L^2(M)}^2,
\]
so the denominator in the definition of $\FR_R(f)$ is exactly $\|P_R f\|_{L^2(M)}$. Dividing the previous inequality by $\|P_R f\|_{L^2(M)}$ yields
\[
\FR_R(f)=\frac{\operatorname{Num}_R(f)}{\|P_R f\|_{L^2(M)}}\ge \frac{1}{C_3 A_R R^{d/2}\sqrt{|E^{C_0/R}|}}.
\]
Setting $c_0=1/C_3$ completes the proof.
\end{proof}

\subsection{Proof of Theorem \ref{theorem:manifold-FR-sandwich}}
\begin{proof}
As before, let
\[
\operatorname{Num}_R(f) = \sum_{\lambda} \left|\psi\left(\frac{\lambda}{R}\right)\right| \|E_\lambda f\|_{L^2(M)},
\qquad
D_R(f) = \left( \sum_{\lambda} \left|\psi\left(\frac{\lambda}{R}\right)\right|^2 \|E_\lambda f\|_{L^2(M)}^2 \right)^{1/2},
\]
so that
\[
\FR_R(f) = \frac{\operatorname{Num}_R(f)}{D_R(f)}.
\]

To prove (i), assume the $L^1$ spectral concentration hypothesis. This means that there exists $\eta \in (0,1)$ such that
\[
\sum_{\substack{\lambda : \left|\psi\left(\frac{\lambda}{R}\right)\right| \ne 0 \\ \lambda \notin \Sigma_R}}
\left|\psi\left(\frac{\lambda}{R}\right)\right| \|E_\lambda f\|_{L^2(M)}
\le
\eta \sum_{\lambda} \left|\psi\left(\frac{\lambda}{R}\right)\right| \|E_\lambda f\|_{L^2(M)}.
\]
Equivalently,
\[
\sum_{\lambda \in \Sigma_R} \left|\psi\left(\frac{\lambda}{R}\right)\right| \|E_\lambda f\|_{L^2(M)} \ge (1 - \eta) \operatorname{Num}_R(f).
\]

On the other hand, by Cauchy--Schwarz on the finite set
\[
\left\{\lambda \in \Sigma_R : \left|\psi\left(\frac{\lambda}{R}\right)\right| \ne 0 \right\},
\]
we have
\[
\sum_{\lambda \in \Sigma_R} \left|\psi\left(\frac{\lambda}{R}\right)\right| \|E_\lambda f\|_{L^2(M)}
\le
\left( \sum_{\lambda \in \Sigma_R} 1 \right)^{1/2}
\left( \sum_{\lambda \in \Sigma_R} \left|\psi\left(\frac{\lambda}{R}\right)\right|^2 \|E_\lambda f\|_{L^2(M)}^2 \right)^{1/2}.
\]
By definition,
\[
\sum_{\lambda \in \Sigma_R} 1 = M_R(\Sigma_R),
\]
and
\[
\sum_{\lambda \in \Sigma_R} \left|\psi\left(\frac{\lambda}{R}\right)\right|^2 \|E_\lambda f\|_{L^2(M)}^2
\le
\sum_{\lambda} \left|\psi\left(\frac{\lambda}{R}\right)\right|^2 \|E_\lambda f\|_{L^2(M)}^2
=
D_R(f)^2.
\]
Therefore
\[
\sum_{\lambda \in \Sigma_R} \left|\psi\left(\frac{\lambda}{R}\right)\right| \|E_\lambda f\|_{L^2(M)}
\le
\sqrt{M_R(\Sigma_R)}\, D_R(f).
\]

Combining this with the lower bound coming from the concentration hypothesis, we obtain
\[
(1 - \eta) \operatorname{Num}_R(f) \le \sqrt{M_R(\Sigma_R)}\, D_R(f).
\]
If $\operatorname{Num}_R(f) = 0$, then $P_R f \equiv 0$ and the claimed inequality is trivial. Otherwise, we divide by $D_R(f)$ and by $(1 - \eta)$ to obtain
\[
\FR_R(f) = \frac{\operatorname{Num}_R(f)}{D_R(f)} \le \frac{\sqrt{M_R(\Sigma_R)}}{1 - \eta},
\]
which is exactly the upper bound asserted in part (i).

To prove (ii), we combine this with the lower bound from Theorem \ref{theorem:localFRlowerbound}. From Theorem \ref{theorem:localFRlowerbound}, we already know that there exists a constant $c_0 > 0$ such that
\[
\FR_R(f) \ge \frac{c_0}{A_R R^{d/2}\sqrt{|E^{C_0/R}|}}.
\]
From (i), we have
\[
\FR_R(f) \le \frac{\sqrt{M_R(\Sigma_R)}}{1 - \eta}.
\]
Putting the two inequalities together gives
\[
\frac{c_0}{A_R R^{d/2}\sqrt{|E^{C_0/R}|}} \le \frac{\sqrt{M_R(\Sigma_R)}}{1 - \eta}.
\]
Rearranging yields,
\[
M_R(\Sigma_R) \, |E^{C_0/R}|
\ge \frac{c_0^2}{A_R^2 R^d} (1 - \eta)^2,
\]
which is exactly the lower bound stated in part (ii). This completes the proof.
\end{proof}

\section{Open problems and further directions}

The results presented in this paper suggest several natural questions for future investigation:

\begin{enumerate}
\item \textbf{Replace Minkowski control by Frostman-type measures/Hausdorff dimension.}  The proof of our   result  in Theorem \ref{main}  uses neighborhood volume growth. A natural extension is: assume 
$u$ satisfies a Frostman condition $|u|(B(x,r) \lesssim r^k$,
 and ask for analogous spectral $\ell^p$ synthesis thresholds. 
 \item \textbf{Optimality beyond sharp thresholds.} 
While we have identified sharp thresholds for spectral synthesis on tori and maximal rigidity on spheres, the optimal integrability exponent $p$ for general compact manifolds remains unclear. Under what geometric conditions on $(M,g)$ does a finite critical exponent exist? How does it relate to the growth of eigenfunctions or the geodesic structure?

\item \textbf{Regularity assumptions.} 
Remark~\ref{remark:regularity-sphere} notes that our proof for the sphere uses the smoothness of the measure $u$. Can Corollary~\ref{cor:sphere} be extended to $u \in L^1_{\mathrm{loc}}(S^d)$ or even to less regular distributions? The dependence on derivatives in the Kuznecov sum formula approach suggests this may require new techniques. The smoothness assumption enters through the use of the Kuznecov formula, whose proof
relies on stationary phase expansions. In particular, quantitative control requires bounds
on finitely many derivatives of the test function, and the resulting constants depend on
these derivatives.

\item \textbf{Refined stability estimates.} 
Theorem~\ref{thm:stability} provides quantitative decay for low-frequency reconstructions. Can these estimates be made uniform over families of manifolds, or can they be sharpened under additional curvature conditions? A more precise understanding of the constant $C(M,g,\psi,C_E,p)$ would be valuable for applications.

\item \textbf{Fourier ratio and approximation for non-band-limited functions.} 
Theorem~\ref{theorem:manifoldfourierratio} shows that small $\FR(f)$ implies good approximation by short spectral sums, and Theorem~\ref{thm:FR-converse-bandlimited} gives a partial converse under band-limitation. Is there a meaningful converse without the band-limiting hypothesis? Can one characterize the class of functions on $M$ with $\FR(f) \lesssim \sqrt{k}$ in terms of their spatial concentration?

\item \textbf{Uncertainty principles for specific manifolds.} 
Theorem~\ref{theorem:manifold-FR-sandwich} gives a general scale-dependent uncertainty principle. For specific manifolds (e.g., spheres, tori, hyperbolic surfaces), can the constants $c_0, C_0$ and the growth parameter $A_R$ be computed explicitly? Such computations would yield concrete, numerically testable uncertainty relations.

\item \textbf{Applications to signal recovery and compressed sensing.} 
As noted in Remark~\ref{rem:conceptual-picture}, the stability estimates suggest applications to inverse problems on manifolds. Can the synthesis principles developed here be used to design efficient sampling schemes or recovery algorithms for signals with thin support on manifolds?

\item \textbf{Extensions to non-compact or singular settings.} 
The compactness assumption is used crucially in the spectral theory and Weyl law. Can similar synthesis phenomena be established for certain non-compact manifolds (e.g., hyperbolic spaces) or for manifolds with boundary? What about singular spaces like graphs or fractal sets equipped with a Laplacian?
\end{enumerate}

We plan to investigate some of these questions in subsequent work.

\appendix

\section{Geometric and spectral preliminaries}\label{app:geometry}

\subsection{Basic facts regarding Riemannian manifolds}

Let $M$ be a smooth $d$-dimensional manifold that is compact and without boundary, and let $g$ be a Riemannian metric tensor on $M$. The metric tensor $g$ is locally represented as a real-symmetric, positive-definite $d \times d$ matrix with entries that are smooth functions of $x$,
\[
    g(x) = \begin{bmatrix}
        g_{11}(x) & g_{12}(x) & \cdots & g_{1d}(x) \\
        g_{21}(x) & g_{22}(x) & \cdots & g_{2d}(x) \\
        \vdots & \vdots & & \vdots \\
        g_{d1}(x) & g_{d2}(x) & \cdots & g_{dd}(x)
    \end{bmatrix}
\]
where, if $\partial_i = \frac{\partial}{\partial x_i}$, we have
\[
    g_{ij}(x) = \langle \partial_i, \partial_j \rangle_{g(x)}.
\]

The natural volume density on a Riemannian manifold $(M,g)$ is given locally by
\[
dV_g(x) = |g(x)|^{1/2}\,dx,
\]
where $|g(x)|$ denotes the determinant of $g(x)$.

\vskip.125in

There is a natural generalization of the Euclidean Laplacian to $(M,g)$, called the Laplace--Beltrami operator, which is expressed in local coordinates by
\[
    \Delta_g f = |g|^{-1/2} \sum_{i,j = 1}^d \partial_i \bigl(|g|^{1/2} g^{ij} \partial_j f\bigr).
\]
Here $g^{ij}$ are the matrix entries of $g^{-1}$. Equivalently, $\Delta_g$ is characterized by the identity
\[
    \int_M (\Delta_g f)\, h \, dV_g = -\int_M \langle \nabla_g f, \nabla_g h \rangle_g \, dV_g,
\]
where
\[
    \nabla_g f(x) = \sum_{i,j = 1}^d g^{ij}(x)\,\partial_i f(x)\,\partial_j
\]
denotes the gradient of $f$. It follows that $\Delta_g$ is self-adjoint and negative semidefinite.

\subsection{Laplace--Beltrami eigenfunctions and eigenvalues}

We now summarize some facts about the spectrum of $\Delta_g$. The material is standard and can be found, for example, in \cite{J02}. There exists an orthonormal basis of $L^2(M,dV_g)$ consisting of eigenfunctions $e_1,e_2,\ldots$ satisfying
\[
    \Delta_g e_j = -\lambda_j^2 e_j, \qquad j = 1,2,\ldots,
\]
where $\lambda_j \ge 0$. We index the eigenfunctions so that
\[
    0 = \lambda_1 \le \lambda_2 \le \lambda_3 \le \cdots \to \infty.
\]
Most of these facts follow by applying the spectral theory of compact self-adjoint operators to the resolvent $(I-\Delta_g)^{-1}$.

\begin{remark}
The choice to use $-\lambda_j^2$ as the eigenvalue for $e_j$ is a matter of convention. This allows us to think of $\lambda_j$ as the ``frequency'' of $e_j$, or rather the frequency of the standing wave solution
\[
    u(t,x) = \cos(t\lambda_j)\,e_j(x)
\]
to the wave equation $(\Delta_g - \partial_t^2)u = 0$.
\end{remark}

Given a function $f\in L^2(M,dV_g)$, the Fourier coefficients of $f$ are $\widehat f(j) = \langle f, e_j \rangle$, where
\[
    \langle f, e_j \rangle = \int_M f(x)\,\overline{e_j(x)} \, dV_g(x).
\]
If $f \in L^2(M,dV_g)$, then
\[
    f = \sum_{j = 1}^\infty \widehat f(j)\, e_j
\]
with convergence in $L^2(M,dV_g)$; if $f$ is sufficiently smooth, then the sum converges uniformly, and equality holds pointwise.

\vskip.125in

A fundamental problem in harmonic analysis on manifolds is to estimate the number of eigenvalues (with multiplicity) below a threshold, i.e.\ to obtain asymptotics for the Weyl counting function
\[
    N(\lambda) = \#\{j \in \mathbb N : \lambda_j \le \lambda \}.
\]
The Weyl law gives
\begin{equation}\label{eq:Weyl-law}
    N(\lambda) = (2\pi)^{-d}\,|M|\,|B|\,\lambda^d + O(\lambda^{d-1}),
\end{equation}
where $|M|$ denotes the Riemannian volume of $M$ and $|B|$ denotes the volume of the unit ball in $\mathbb R^d$. The remainder is sharp in the sense that it cannot be improved for the sphere $S^d$. Many manifolds enjoy an improved Weyl remainder estimate \cite{DG75}.

\end{document}